\documentclass[10pt]{amsart}

\usepackage{mathtools}
\usepackage{amssymb,amsfonts,amsmath}
\usepackage{hyperref}
\usepackage{amsthm}

\usepackage{booktabs}
\usepackage{bm}

\usepackage{graphicx, verbatim, centernot}
\usepackage{caption}
\usepackage{xcolor}

\usepackage{multicol}
\usepackage[shortlabels]{enumitem}

\usepackage{dsfont}

\allowdisplaybreaks
\usepackage[left=1in,right=1in,top=1in,bottom=1in]{geometry}

\usepackage[capitalize]{cleveref}

\newtheorem{theorem}{Theorem}[section]
\newtheorem*{theorem*}{Theorem}

\newtheorem{lemma}[theorem]{Lemma}

\newtheorem*{claim*}{Claim}

\newcommand{\one}{\boldsymbol{1}}

\newcommand{\E}{\mathbb{E}}

\newcommand{\R}{\mathbb{R}}

\newcommand{\Var}{{\operatorname{Var}}}

\newcommand{\sigmabf}{ \boldsymbol{\sigma} }

\newcommand{\rplus}[1]{ r_{#1}(\sigmabf) }

\newcommand{\taubf}{ \boldsymbol{\tau} }

\AtBeginDocument{\pagestyle{plain}}

\begin{document}
	
	\title{A simple proof of rapid mixing on random regular graphs beyond uniqueness}
	
	\author{Andreas G\"obel}
	\address{Hasso Plattner Institute, University of Potsdam, Germany}
	\email{andreas.goebel@hpi.de}
	\author{Matthew Jenssen}
	\address{King's College London, Department of Mathematics}
	\email{matthew.jenssen@kcl.ac.uk}
	\author{Marcus Michelen}
	\address{Northwestern University, Department of Mathematics}
	\email{michelen@northwestern.edu}
	\author{Marcus Pappik}
	\address{Hasso Plattner Institute, University of Potsdam, Germany}
	\email{marcus.pappik@hpi.de}
	\author{Will Perkins}
	\address{Georgia Institute of Technology, School of Computer Science}
	\email{math@willperkins.org}
	\author{Leon Schiller}
	\address{Hasso Plattner Institute, University of Potsdam, Germany}
	\email{leon.schiller@hpi.de}
	
	\date{}

	\begin{abstract}
		A recent breakthrough of Chen, Chen, Chen, Yin, and Zhang  shows rapid mixing for Glauber dynamics for the hard-core model on random regular graphs beyond the tree uniqueness threshold.  Their approach builds upon the literature of various local-to-global techniques and applies to a more general setting of 
		discrete distributions 
		supported on downward-closed set families.  We give a short and self-contained proof via a Bochner--Bakry--\'{E}mery approach and directly show a Poincar\'e inequality by expanding the Dirichlet form in terms of the $L^2$-norm of the generator applied to a test function and eliminating a sum of squares term.   Our proof is a streamlined version of an argument of Kondratiev, Kuna, and Ohlerich 
		used to study spatial birth-and-death dynamics for Gibbs 
		point processes in the continuum, which we adapt to the discrete setting.

	\end{abstract}
	
	\maketitle
	
	\section{Introduction}
	
	Sampling from a high-dimensional probability distribution supported on a combinatorial 
	structure (e.g. independent sets of a graph, matchings, bases of a matroid) is a central problem in statistical physics and theoretical computer science. One of the earliest and most successful techniques developed for this task is the Markov chain Monte Carlo (MCMC) method. 
	A canonical and widely studied Markov chain is the 
	\emph{Glauber dynamics} (or Gibbs sampler), in which in single coordinate is updated at each step conditioned on the current state of the remaining coordinates.  A fundamental question is 
	to determine when this Markov chain mixes rapidly.
	
	A central  example is the hard-core model on a finite graph $G$.  Fixing an activity parameter $\lambda>0$, the \emph{hard-core model} is the distribution $\mu_{G,\lambda}$ on independent sets 
	$I$ of $G$ defined by 
	\[
	\mu_{G,\lambda}(I)= \frac{\lambda^{|I|}}{Z_G(\lambda)}\, ,
	\]
	where $Z_G(\lambda)=\sum_I \lambda^{|I|}$.

	One reason the hard-core model is so central to this topic is a beautiful connection between statistical physics phase transitions, computational complexity, and rapid mixing of Markov chains.
	For the class of graphs of maximum degree $\Delta$, the worst-case computational complexity of sampling from the hard-core model is now well understood, with a sharp transition at the  threshold $\lambda_c(\Delta) = (\Delta-1)^{\Delta-1}/(\Delta-2)^\Delta \sim e/\Delta$, which marks the uniqueness/non-uniqueness phase transition of the hard-core model on the infinite $\Delta$-regular tree.  Weitz~\cite{Wei06} gave a polynomial-time  sampling algorithm for $\lambda < \lambda_c(\Delta)$ based on the method of correlation decay on computational trees; then a long line of work culminating in~\cite{ALO20, CLV23,CLV21}  showed that the 
	Glauber dynamics mixes in the optimal $O(n\log n)$ time. At $\lambda=\lambda_c(\Delta)$ the worst-case mixing time is polynomial~\cite{chen2025rapid-unique}. 
	On the other hand, when $\lambda> \lambda_c(\Delta)$, the Glauber dynamics mixes exponentially slowly in the 
	worst case such as on typical random $\Delta$-regular bipartite graphs~\cite{MWW09}; 
	Sly~\cite{Sly10} used random bipartite graphs as gadgets to prove that no efficient sampling algorithm  exists for $\lambda > \lambda_c(\Delta)$ unless $\mathsf{NP}=\mathsf{RP}$.

	Despite the worst-case slow mixing for $\lambda > \lambda_c(\Delta)$, much 
	better mixing behavior is expected for many instances, and in particular on random (not bipartite) instances. It is conjectured that the Glauber dynamics mixes rapidly  on random $\Delta$-regular graphs  all the way up to the \textit{reconstruction threshold} 
	$\lambda_r(\Delta) = \log ^{2+o(1)} (\Delta) \gg  \lambda_c(\Delta)$~\cite{brightwell2004second,bhatnagar2016decay}.   In a recent breakthrough, Chen, Chen, Chen, Yin, and Zhang~\cite{chen2025rapid} established rapid mixing 
	for $\lambda = O(1/\sqrt{\Delta})$ on random $\Delta$-regular graphs using a lower bound on the most negative graph eigenvalue, which lies significantly beyond the uniqueness threshold (but still far short of $\lambda_r$).  Their result is deduced 
	from a general criterion for rapid mixing of the Glauber dynamics of any distribution supported on a downward-closed set family: roughly, the criterion asks that a certain matrix 
	of local pairwise correlations be spectrally bounded after conditioning on any partial configuration. The proof combines several recent technical ingredients: a trickle-down theorem for field dynamics, spectral stability, and a comparison between field dynamics and Glauber dynamics.  While the bound on $\lambda$ is a factor $\Delta^{1/2 +o(1)}$ below  the reconstruction threshold, the average density of independent sets at the achieved bound is only a factor $2$ away from the density at the reconstruction threshold (roughly $\frac{1}{2} \frac{\log \Delta}{\Delta}$ versus $\frac{\log \Delta}{\Delta}$); see the remarks after~\cite[Theorem 1.1]{chen2025rapid}.

	A similar Markov chain 
	was studied more than a decade earlier by Kondratiev, Kuna, and Ohlerich~\cite{kondratiev2013spectral} in the setting of Gibbs point processes on $\mathbb{R}^d$. The relevant 
	dynamics there is a birth-and-death process (points appearing and disappearing in random locations at random times) whose rates are governed by the Papangelou intensity, the continuum analogue of the marginal ratios appearing in~\cite{chen2025rapid}. Their main result is a sufficient condition for a lower bound on the
	spectral gap of the infinitesimal generator that is structurally very similar to the criterion of~\cite{chen2025rapid}: positive-semidefiniteness of a kernel built from the Papangelou intensity, in place of the matrix built 
	from the discrete marginal ratios.

	The parallel between the two results is striking, and the purpose of this note is to make the connection precise.  
	We show that the argument of Kondratiev, Kuna, and Ohlerich adapts directly to this discrete setting, yielding a short, self-contained proof of \cite[Theorem~1.9]{chen2025rapid}: the matrix condition ensuring a spectral gap for distributions on downward-closed set families (which implies the hard-core result on random regular graphs).
	We note that under a strictly stronger matrix condition, \cite[Theorem~1.8]{chen2025rapid} gives a modified log-Sobolev inequality yielding a near-optimal mixing time, and our approach does not recover this.

	The work of Kondratiev, Kuna, and Ohlerich is inspired by a classical route from local curvature estimates to spectral gaps.
	In Riemannian geometry, Bochner's identity relates the dissipation of the
	gradient of a function to a nonnegative second-derivative term and a curvature
	term.  Bakry and \'Emery~\cite{bakry2006diffusions} later reformulated this idea in the language of
	Markov generators, turning it into a general method for proving Poincar\'e and
	log-Sobolev inequalities.  For jump processes, this philosophy was developed in a form closer to the
	present paper by Boudou, Caputo, Dai~Pra, and Posta
	\cite{boudou2006spectral}.  Kondratiev, Kuna, and Ohlerich
	\cite{kondratiev2013spectral} later built on this approach in the setting of continuum
	Gibbs point processes.  
	
	We note that the work \cite{kondratiev2013spectral} is phrased in terms of the so-called \emph{carr\'e du champ} operators $\Gamma$ and $\Gamma_2$.  They write full expansions for the quantities $\Gamma(f,f)$ and $\Gamma_2(f,f)$ and provide ``an analogue of the Bochner--Lichn\'erowicz--Weitzenb\"ock formula.'' 
	We do not introduce the carr\'{e} du champ formalism and instead work directly with the integrated forms of $\Gamma(f,f)$ and $\Gamma_2(f,f)$: the
	Dirichlet form and the variance $\Var(Lf)$ where $L$ is the generator of the birth-death dynamics. 
	We also take advantage of the recursive nature of the discrete setting: we observe that the discrete derivative can be expressed in terms of the generator of the same birth--death chain on the conditional family where a vertex has been pinned, together with an explicit interaction term measuring how pinning a vertex changes the birth rates of other elements. After integration, the conditional-generator term is nonnegative, and the remaining interaction term is precisely the local matrix controlled by the condition of Chen, Chen, Chen, Yin, Zhang~\cite{chen2025rapid}.

	\subsection*{Main results}

	Given a graph $G=(V,E)$ and activity $\lambda$, the \emph{Glauber dynamics} for the hard-core model $\mu=\mu_{G,\lambda}$ is a Markov chain on the independent sets of $G$ defined as follows. In each 
	step, the chain picks an element $v\in V$ uniformly at random and resamples its state (occupied or unoccupied) according to the measure $\mu$ conditioned on the current state of all vertices in $V\backslash\{v\}$.
	The transition matrix $P_{\mathrm{GD}}$ satisfies, for
	$\sigmabf\neq \taubf$,
	\begin{align}\label{eq:Glauber}
		P_{\mathrm{GD}}(\sigmabf,\taubf)
		=
		\begin{cases}
			\displaystyle
			\frac{1}{|V|}\frac{\mu(\taubf)}{\mu(\sigmabf)+\mu(\taubf)},
			&\text{if }\sigmabf\triangle \taubf=\{v\}\text{ for some }v\in V,\\[3mm]
			0,
			&\text{otherwise.}
		\end{cases}
	\end{align}
	The \emph{mixing time} of the Glauber dynamics is defined as
	\[
	T_{\mathrm{mix}}(P_{\mathrm{GD}})
	=
	\max_{\sigmabf}
	\min\left\{
	t\ge 0 \,\middle|\,
	d_{\mathrm{TV}}\bigl(P_{\mathrm{GD}}^t(\sigmabf,\cdot),\mu\bigr)<\frac14
	\right\},
	\]
	where
	\(
	d_{\mathrm{TV}}(\mu,\nu)
	\)
	denotes the total variation distance.
	Our main goal is to give a short, self-contained proof of the following theorem of Chen, Chen, Chen, Yin, and Zhang~\cite[Theorem 1.2]{chen2025rapid}
	
	\begin{theorem} \label{thm:hard-core-model}
		Let $\delta \in (0,1)$ and consider Glauber dynamics for the hard-core model on a graph $G$ with $n$ vertices and activity $\lambda$.  If $$\lambda \leq \frac{1 - \delta}{-\lambda_{\min}(A_G) - 1}$$
		where $\lambda_{\min}(A_G)$ is the minimum eigenvalue of the adjacency matrix $A_G$ of $G$, then $T_{\mathrm{mix}}(P_{\mathrm{GD}}) \leq O(\frac{n^2}{\delta} \log \frac{1}{\lambda})\,.$
	\end{theorem}
	Recall that if $G$ is random $\Delta$-regular graph on $n$ vertices, then $\lambda_{\min}(A_G)=-2\sqrt{\Delta-1} +o_n(1)$ with high probability~\cite{friedman2008proof}, in which case the above theorem applies for $\lambda\leq \frac{1-\delta}{2\sqrt{\Delta-1}-1}$. 
	
	As in \cite{chen2025rapid} we work in the more general setting where $\mu$ is a distribution that is fully supported on a downward-closed set family $\mathcal{X} \subseteq \{0, 1\}^V$ for $V$ finite. Note that we may define the Glauber dynamics for $\mu$ exactly as we did for hard-core model in~\eqref{eq:Glauber}.

	It will be convenient to also work with a \emph{continuous time} version of the Glauber dynamics. We will define dynamics that are reversible for $\mu$ by first specifying jump rates  
	$$
	\rplus{v} \coloneqq \begin{cases} \frac{\mu(\sigmabf \cup \{v\})}{\mu(\sigmabf)} & \text{if } v \notin \sigmabf \text{ and } \sigmabf \cup \{v\}  \in \mathcal{X} \\
		0 & \text{otherwise .}
	\end{cases} \,.
	$$
	For a given configuration $\sigmabf\in \mathcal{X}$ and a function $f:\{0,1\}^V \to \R$ define the operators
	$$
	(D^-_vf)(\sigmabf) = f(\sigmabf \setminus \{v\}) - f(\sigmabf) \quad \text{ and }\quad  (D^+_vf)(\sigmabf) = f(\sigmabf \cup \{v\}) - f(\sigmabf) \,.
	$$
	Continuous-time Glauber dynamics is then defined by the generator $$
	(Lf)(\sigmabf) = \sum_{v\in V} \left(r_v(\sigmabf) (D_{v}^+f)(\sigmabf) + (D_{v}^-f)(\sigmabf)\right) \,.
	$$ 
	The dynamics generated by $L$ may be described as follows: 
	when the chain is at a configuration
	$\sigmabf\in\mathcal X$, each occupied element $v\in \sigmabf$ dies, i.e. is removed from
	$\sigmabf$, at rate $1$.  Each unoccupied element $v\notin \sigmabf$ 
	is born, i.e. is added to $\sigmabf$, at rate
	$r_v(\sigmabf)$.

	For functions $f$ and $g$, we define the inner product $\langle f, g\rangle_\mu := \E_\mu[f(\sigmabf)g(\sigmabf)]$ where $\sigmabf \sim \mu$. The \emph{spectral gap} $\gamma$ of $L$ is defined by \begin{equation}\label{eq:spectral-gap-def}
		\gamma(L) = \inf_{f} \frac{\langle f, -Lf\rangle_\mu}{\text{Var}_\mu(f)}\,.
	\end{equation}
	where the infimum is over non-constant $f:\{0,1\}^V\to\mathbb{R}$.
	We prove the following version of \cite[Theorem~1.9]{chen2025rapid}.  
	\begin{theorem}\label{th:main}
		Let $\mathcal{X} \subseteq \{0,1\}^V$ be a non-empty downward closed set family and suppose $\mu$ is a distribution fully supported on $\mathcal{X}$.  Suppose that for $\delta \in (0,1)$ we have that for all $\sigmabf \in \mathcal{X}$ the matrix $M_{\sigmabf}\in \R^{V\times V}$ given by
		\begin{equation}\label{eq:PSD-assumption}
			M_{\sigmabf}(u,v)=(1 - \delta)r_v(\sigmabf)\one_{u = v} - r_v(\sigmabf)(r_u(\sigmabf\cup \{v\}) - r_u(\sigmabf))  
		\end{equation}
		is positive semidefinite. Then continuous time Glauber dynamics for $\mu$ has spectral gap at least $\delta$.
	\end{theorem}
	
	We note that the matrix condition~\eqref{eq:PSD-assumption} is simply a reparameterisation the condition in \cite[Thm.~1.9]{chen2025rapid}. A standard comparison of Dirichlet forms shows that the spectral gap of the continuous-time Glauber dynamics gives a lower bound for the spectral gap of the discrete-time Glauber dynamics. Indeed, 
	\[
	\langle f,(I-P_{\mathrm{GD}})f\rangle_\mu
	=
	\frac{1}{|V|}
	\sum_{v\in V}
	\E_\mu\left[\frac{r_v(\sigmabf)}{1+r_v(\sigmabf)}
	(D_v^+f(\sigmabf))^2\right] \geq \sum_{v\in V} \frac{
		\E_\mu\left[r_v(\sigmabf)
		(D_v^+f(\sigmabf))^2\right]}{(1+r_{\max})|V|} =  \frac{\langle f,-Lf\rangle_\mu}{(1+r_{\max})|V|} .
	\]
	where $r_{\max}=\max_{\sigmabf\in\mathcal X, v\in V} r_v(\sigmabf)$ and the last equality is due to Lemma~\ref{lem:ibp} below.  
	It follows from~\eqref{eq:spectral-gap-def} that  
	\begin{align}\label{eq:gapcomp}
		\gamma(P_{\mathrm{GD}}-I)
		\ge
		\frac{1}{(1+r_{\max})|V|}\gamma(L) \,.
	\end{align}

	Our strategy for proving \cref{th:main} is to streamline  the proof of \cite[Thm.~4.2]{kondratiev2013spectral} which shows an analogous bound on the spectral gap for spatial birth-death dynamics of Gibbs point processes. Before turning to the proof, we deduce \cref{thm:hard-core-model} from \cref{th:main}.

	\begin{proof}[Proof of \cref{thm:hard-core-model}]
		We will apply \cref{th:main} with $\mu=\mu_{G,\lambda}$, the hard-core model on $G$ at activity $\lambda$.  Here $\mathcal{X}$ is the set of independent sets of $G$.  For an independent set $\sigmabf \in \mathcal{X}$, we let $G^{\sigmabf}$ denote the graph that remains when we remove $\sigmabf$ and its neighbors from $G$.  We note that $M_{\sigmabf} = \lambda\left(I(1 - \delta) + \lambda(I + A_{G^{\sigmabf}}) \right)$ where $A_{G^{\sigmabf}}$ is the adjacency matrix of $G^{\sigmabf}$.  By eigenvalue interlacing we note that $\lambda_{\min}(A_{G^{\sigmabf}}) \geq \lambda_{\min}(A_G)$.  Thus by \cref{th:main} we have $\gamma(L) \geq \delta$ and so $\gamma(P_{\mathrm{GD}}) \geq \delta/((1+\lambda)n)$ by~\eqref{eq:gapcomp}. 
		This shows $T_{\mathrm{mix}}(P_{\mathrm{GD}}) \leq (1+\lambda)n\delta^{-1}\log\left(4/\mu_{\min}\right)$ where $\mu_{\min}=\min_{\sigmabf} \mu(\sigmabf)$.  Bounding $\log(1/\mu_{\min}) \leq O(n \log(1/\lambda))$ completes the proof.
	\end{proof}

	\section{Proof of \texorpdfstring{\cref{th:main}}{Theorem 1.2}}
	It will be convenient to work with the following equivalent definition of the spectral gap. 
	\[
	\gamma(L) = \inf_{f} \frac{\langle Lf,Lf\rangle_\mu}{\langle f, -Lf\rangle_\mu}\,.
	\]
	The equivalence to~\eqref{eq:spectral-gap-def} follows from self-adjointness of $L$ which we prove now. More generally, we establish the following integration-by-parts identity. 
	
	\begin{lemma}\label{lem:ibp}
		For all $f,g:\{0,1\}^V\to\mathbb R$,
		\begin{equation}\label{eq:ibp}
			-\langle f,Lg\rangle_\mu
			=
			\sum_{v\in V}\E_\mu\left[r_v(\sigmabf)D_v^+f(\sigmabf)\,D_v^+g(\sigmabf)
			\right].
		\end{equation}
		In particular, $L$ is self-adjoint with respect to $\langle \cdot,\cdot\rangle_\mu$.
	\end{lemma}
	
	\begin{proof}
		Fix $v\in V$. Since $D_v^-g(\sigmabf)=0$ whenever $v\notin \sigmabf$ and $D_v^+g(\sigmabf)=0$ whenever $v\in \sigmabf$ we have
		\begin{align*}
			&\E_\mu\left[f(\sigmabf)\Bigl(r_v(\sigmabf)D_v^+g(\sigmabf)+D_v^-g(\sigmabf)\Bigr)\right] \\
			&=
			\sum_{\substack{\sigmabf\in\mathcal X\\ v\notin \sigmabf}}
			\mu(\sigmabf)f(\sigmabf)r_v(\sigmabf)\bigl(g(\sigmabf\cup\{v\})-g(\sigmabf)\bigr)
			+
			\sum_{\substack{\sigmabf\in\mathcal X\\ v\in \sigmabf}}
			\mu(\sigmabf)f(\sigmabf)\bigl(g(\sigmabf\setminus\{v\})-g(\sigmabf)\bigr)
		\end{align*}
		In the second sum, set $\taubf=\sigmabf\backslash\{v\}$. 
		Since
		\(
		\mu(\sigmabf)=\mu(\taubf \cup\{v\})=\mu(\taubf)r_v(\taubf)
		\)
		the above equals
		\begin{align*}
			&\sum_{\substack{\sigmabf\in\mathcal X\\ v\notin \sigmabf}}
			\mu(\sigmabf)f(\sigmabf)r_v(\sigmabf)\bigl(g(\sigmabf\cup\{v\})-g(\sigmabf)\bigr)
			+
			\sum_{\substack{\taubf\in\mathcal X\\ v\notin \taubf}}
			\mu(\taubf)r_v(\taubf) f(\taubf \cup\{v\})\bigl(g(\taubf)-g(\taubf \cup\{v\})\bigr) \\
			&=
			-\sum_{\sigmabf\in\mathcal X}
			\mu(\sigmabf)r_v(\sigmabf) D_v^+f(\sigmabf) D_v^+g(\sigmabf).
		\end{align*}
		Summing over $v\in V$ proves \eqref{eq:ibp}.
		The right-hand side of \eqref{eq:ibp} is symmetric in $f$ and $g$, so $L$ is self-adjoint. 
	\end{proof}

	For each $v\in V$, let
	\[
	\mathcal X^{(v)}:=\{\sigmabf\subseteq V\setminus\{v\}:\ \sigmabf\cup\{v\}\in \mathcal X\}
	\]
	and define the conditioned measure
	\[
	\mu^{(v)}(\sigmabf)
	:=
	\frac{\mu(\sigmabf\cup\{v\})}
	{\sum_{\taubf\in\mathcal X^{(v)}}\mu(\taubf\cup\{v\})},
	\qquad \sigmabf\in \mathcal X^{(v)}.
	\]
	The associated Glauber generator on $\mathcal X^{(v)}$ is
	\[
	(L^{(v)}h)(\sigmabf)
	:=
	\sum_{u\in V\setminus\{v\}}
	\left(
	r_u(\sigmabf\cup\{v\})\,D_u^+h(\sigmabf)+D_u^-h(\sigmabf)
	\right).
	\]
	Note that $(\mu^{(v)},L^{(v)})$ is of the same form as $(\mu,L)$, so Lemma~\ref{lem:ibp}
	applies to $L^{(v)}$ as well.
	
	The following `commutator identity' is key to our proof.
	
	\begin{lemma}\label{lem:commutator}
		Fix $v\in V$. For every $\sigmabf\in\mathcal X$ with $v\notin \sigmabf$,
		\begin{equation}\label{eq:commutator}
			D_v^+Lf(\sigmabf)
			=
			L^{(v)}(D_v^+f)(\sigmabf)
			-
			D_v^+f(\sigmabf)
			+
			\sum_{u\in V}(D_v^+r_u)(\sigmabf)\,D_u^+f(\sigmabf).
		\end{equation}
	\end{lemma}
	
	\begin{proof}
		Fix $\sigmabf\in\mathcal X$ with $v\notin \sigmabf$. Expanding $D_v^+L f(\sigmabf)$ gives
		\begin{align}\label{eq:DplusL}
			D_v^+Lf(\sigmabf)
			&=
			D_v^+(r_vD_v^+f)(\sigmabf)+D_v^+D_v^-f(\sigmabf)+
			\sum_{\substack{u\in V\backslash\{v\}}}
			\Bigl(
			D_v^+(r_uD_u^+f)(\sigmabf)+D_v^+D_u^-f(\sigmabf)
			\Bigr)\, .
		\end{align}
		For $u\neq v$, observe that the discrete derivatives commute:
		\[
		D_v^+D_u^-f=D_u^-D_v^+f,
		\qquad
		D_v^+D_u^+f=D_u^+D_v^+f\, .
		\]
		Note also that since $v\notin \sigmabf$,
		\[
		D_v^+D_v^-f(\sigmabf)=-D_v^+f(\sigmabf)\, .
		\]
		For all $u,v$ we have
		\[
		D_v^+(r_uD_u^+f)(\sigmabf)
		=
		(D_v^+r_u)(\sigmabf)\,D_u^+f(\sigmabf)
		+
		r_u(\sigmabf\cup\{v\})\,D_u^+D_v^+f(\sigmabf).
		\]
		Setting $u=v$ in the above and recalling that $r_v(\taubf)=0$ if $v\in \taubf$ we obtain 
		\[
		D_v^+(r_vD_v^+f)(\sigmabf)=-r_v(\sigmabf)\,D_v^+f(\sigmabf)\, .
		\]
		Substituting these identities into~\eqref{eq:DplusL} yields \eqref{eq:commutator}.
	\end{proof}
	
	We deduce the following key Bochner-type inequality. 
	
	\begin{lemma}\label{lem:bochner}
		For every $f:\{0,1\}^V\to\mathbb R$,
		\begin{align*}
			\langle Lf,Lf\rangle_\mu
			\ge
			-\langle f,Lf\rangle_\mu
			-
			\E_\mu\!\left[
			\sum_{u,v\in V}
			r_v(\sigmabf)\,(D_v^+r_u)(\sigmabf)\,D_v^+f(\sigmabf)\,D_u^+f(\sigmabf)
			\right].
		\end{align*}
	\end{lemma}
	
	\begin{proof}
		By Lemma~\ref{lem:ibp},
		\[
		\langle Lf,Lf\rangle_\mu
		=
		\langle f,L^2f\rangle_\mu
		=
		-\sum_{v\in V}
		\E_\mu\!\left[
		r_v(\sigmabf)\,D_v^+f(\sigmabf)\,D_v^+Lf(\sigmabf)
		\right].
		\]
		Applying Lemma~\ref{lem:commutator} yields
		\begin{align*}
			\langle Lf,Lf\rangle_\mu
			&=
			\sum_{v\in V}
			\E_\mu\!\left[r_v(\sigmabf)\bigl(D_v^+f(\sigmabf)\bigr)^2\right] 
			-
			\E_\mu\!\left[
			\sum_{u,v\in V}
			r_v(\sigmabf)\,(D_v^+r_u)(\sigmabf)\,D_v^+f(\sigmabf)\,D_u^+f(\sigmabf)
			\right] \\
			&\hspace{0.5 cm}
			-
			\sum_{v\in V}
			\E_\mu\!\left[
			r_v(\sigmabf)\,D_v^+f(\sigmabf)\,L^{(v)}(D_v^+f)(\sigmabf)
			\right].
		\end{align*}
		The first sum is exactly $-\langle f,Lf\rangle_\mu$ by Lemma~\ref{lem:ibp}. It remains to show that the last sum is non-negative. For this set $h=D_v^+f$ and note that 
		\begin{align*}
			\E_\mu\left[
			r_v(\sigmabf)\,h(\sigmabf)\,L^{(v)}h(\sigmabf)
			\right]
			=
			\sum_{\sigmabf\in \mathcal{X}} \mu(\sigmabf)r_v(\sigmabf)h(\sigmabf)L^{(v)}h(\sigmabf)
			&=
			\sum_{\sigmabf\in \mathcal{X}^{(v)}} \mu(\sigmabf \cup \{v\})h(\sigmabf)L^{(v)}h(\sigmabf)\, .
		\end{align*}
		Writing $Z=\sum_{\taubf\in\mathcal X^{(v)}}\mu(\taubf\cup\{v\})$ and 
		applying Lemma~\ref{lem:ibp} to $(\mu^{(v)},L^{(v)})$ we see that the above is equal to
		\[
		Z \cdot \E_{\mu^{(v)}}
		\left[ h(\sigmabf) L^{(v)}h(\sigmabf)
		\right]=-Z\sum_{u \in V\backslash\{v\}} 
		\E_{\mu^{(v)}}\left[ 
		r_u(\sigmabf \cup\{v\})\bigl(D_u^+h(\sigmabf)\bigr)^2
		\right]\leq 0\, . \qedhere
		\]
	\end{proof}
	
	\begin{proof}[Proof of \cref{th:main}]
		Fix $f:\{0,1\}^V\to\mathbb R$, and for each configuration $\sigmabf\in \mathcal X$ and $v\in V$ define
		\(
		\psi_{\sigmabf}(v):=D_v^+f(\sigmabf).
		\)
		By Lemmas~\ref{lem:ibp} and~\ref{lem:bochner},
		\begin{align*}
			\langle Lf,Lf\rangle_\mu+\delta\langle f,Lf\rangle_\mu
			&\ge
			\E_\mu\!\left[
			\sum_{u,v\in V}
			\Bigl(
			(1-\delta)r_v(\sigmabf)\mathbf 1_{\{u=v\}}
			-
			r_v(\sigmabf)(D_v^+r_u)(\sigmabf)
			\Bigr)
			\psi_{\sigmabf}(u)\psi_{\sigmabf}(v)
			\right].
		\end{align*}
		Note  that the matrix inside the expectation is exactly the matrix from \eqref{eq:PSD-assumption}. By
		assumption, its quadratic form is nonnegative for every $\sigmabf$, and so this shows.
		\[
		\langle Lf,Lf\rangle_\mu+ \delta\langle f,Lf\rangle_\mu \geq 0\,. \qedhere 
		\]
	\end{proof}
	
	\section*{Acknowledgments}
	A.G. is funded by the Postdoc Network Brandenburg.
	M.J.\ is supported by a UK Research and Innovation Future Leaders Fellowship MR/W007320/2.  M.M.\ is supported in part by NSF grants DMS-2336788 and DMS-2246624. 
	M.P.\ is funded by the Deutsche Forschungsgemeinschaft (DFG, German Research Foundation) -- project number 390859508.
	W.P.\ supported in part by NSF grant DMS-2348743.

	\bibliography{references.bib}
	\bibliographystyle{abbrv}
	
\end{document}